\documentclass[preprint,12pt]{elsarticle}

\usepackage{amssymb}
\usepackage{hyperref}
\usepackage{amsmath}
\usepackage{enumerate}
\usepackage{amsthm}
\usepackage{amsfonts}
\usepackage{subeqnarray}
\usepackage{mathrsfs}
\usepackage{latexsym}
\usepackage{dsfont}
\usepackage{mathtools} 

\usepackage[margin=2cm]{geometry}
\allowdisplaybreaks[4]
\newcommand{\inj}{\hookrightarrow}
\newtheorem{theorem}{Theorem}[section]
\newtheorem{remark}{Remark}[section]

\newtheorem{proposition}{Proposition}[section]
\newtheorem{lemma}{Lemma}[section]

\journal{Elsevier}

\begin{document}

\begin{frontmatter}



\title{Well-posedness of a Debye type 
system endowed with a full wave equation.}


\author{Arnaud Heibig
\footnote{Corresponding author.  
E-mail adress: arnaud.heibig@insa-lyon.fr Fax: +33 472438529} 
}
\address{Universit\'e de Lyon, Institut Camille Jordan, 
INSA-Lyon,  B\^at. Leonard de Vinci No. 401, 21 Avenue Jean Capelle, 
F-69621, Villeurbanne, France.}

\begin{abstract}
We prove well-posedness for a transport-diffusion problem coupled with a wave
equation for the potential. We assume
that the initial data are small. A bilinear form in the
spirit of Kato's proof for the Navier-Stokes
equations is used, coupled with 
suitable estimates 
in Chemin-Lerner spaces. In the one dimensional case,
we get well-posedness for arbitrarily large initial data by using
Gagliardo-Nirenberg inequalities.
\end{abstract}

\begin{keyword}
Transport-diffusion equation \sep 
wave equation \sep  
Debye system \sep Chemin-Lerner spaces \sep
Gagliardo-Nirenberg inequalities.

%%%\MSC 35Q35 \sep \MSC 35Q84 \sep \MSC 35Q30
\end{keyword}

\end{frontmatter}


\section{Introduction.}\label{intro}
Transport-diffusion equations
have a vast phenomenology and have been widely studied. See, among others,
\cite{bd}, \cite{bwn}, 
\cite{kn}, 
\cite{w} in the case
of the semi-conductor theory,
and \cite{cs} in the case of Fokker-Planck equations. The goal of this
note is to prove existence and
uniqueness of the solution for
a modify semi-conductor equation.

In order to simplify the presentation,
we restrict to the case of a single
electrical charge. The novelty of
our equations is that we replace the
Poisson equation on the potential
by a wave equation. This is a quite natural change, 
since the electric charge itself 
depends on the time. From a mathematical
point of view, switching from a Poisson
equation to  a wave equation roughly
amounts to the loss of one derivative
in the estimates on the potential. Moreover,
it seems that one is bound to work in 
$L^p_t$ spaces with $1\leq p \leq2$ due 
to the usual Strichartz estimates. 

In
this paper, we prove the existence of a 
mild
solution
in Chemin-Lerner spaces
$\tilde{L}^1
(0, T, \\{\dot{H}}^{n/2-1}(\mathbb{R}^{n}))$.
We first restrict to the case of small initial data ($n\geq 2$), 
and use a variant of the Picard fixed point
theorem as in the proof
of Kato's and Chemin's theorems for the Navier-Stokes (and related)
equations. 
See \cite{we}, 
\cite{ka}, \cite{ch} and also
\cite{le}, \cite{bcd}. In particular, we work in homogeneous 
Sobolev spaces 
in order to get
$T$-independent estimates for the heat equation. Note also that our bilinear
form depends on a nonlocal term,
given as the solution of the wave equation
on the 
potential.

In the case $n=1$, well posedness is 
established for arbitrary large
initial data (section $\ref{prelim}$). Local well posedness is 
obtained as
in section $\ref{exit}$. The global existence
is proved by combining the usual 
$L^1$ estimate with 
a Gagliardo-Nirenberg inequality, in the spirit of \cite{bwn}.

%%%%%%%%%%%%%%%%%%%%%%%%%%%%%%%%%%%%% %
%%%%%%%%%%%%%%%%%%%%%%%%%%%%%%%%%%%%%%
\section{Equations and preliminary results.}\label{mappa}
%%%%%%%%%%%%%%%%%%%%%%%%%%%%%%%%%%%%%%
%%%%%%%%%%%%%%%%%%%%%%%%%%%%%%%%%%%%%%
%%%
%%%%%%%%%%%%%%%%%%%%%%%%%%%%%%%%%%%%%%%%%%%%%%%%%%%%%%%%%%%%%%%

%a) Notations. 
We begin with some notations. In this section $n\geq 2$,
$T>0$, and $s<n/2$ are given.  
The homogeneous Sobolev 
spaces
${\dot{H}}^s
(\mathbb{R}^n)$ are often denoted by 
${\dot{H}}^s$. For $p\geq 1$,
we also
use the Chemin-Lerner spaces 
$\tilde{L}^p
(0, T, {\dot{H}}^s(\mathbb{R}^n))
= \tilde{L}^p
(0, T, {\dot{B}}^{s}_{2, 2}(\mathbb{R}^n))$,
or simply
$\tilde{L}^p_T
({\dot{H}}^s)$.
Recall that a distribution $f
\in\mathscr{S}'\big(]0, T[\times 
\mathbb{R}^n)$ belongs to the space $\tilde{L}^p_T
({\dot{H}}^s)$ 
iff  
${\dot{S}}_{j}f \rightarrow
 0$ in $\mathscr{S}'$ for $j \rightarrow 
 -\infty$,
and
$\Vert f 
\Vert_{\tilde{L}^p_T
({\dot{H}}^s)}
:=
\Vert
(2^{js}\Vert {\dot{\Delta}}_{j}f 
\Vert_{L^p_T(L^{2})})_{j\in \mathbb{Z}}
\Vert_{l^2(\mathbb{Z})}
< \infty$. Here, 
${\dot{S}}_{j}f$ and
 ${\dot{\Delta}}_{j}f$
 are respectively the low frequency
 cut-off and the homogeneous dyadic
 block defined by the usual
 Paley-Littlewood decomposition.
 See \cite{bcd} p.98 for details. Last, we write
 $\nabla$ for the (spatial) gradient, 
 $div$
 for the divergence and  $\Delta = div\nabla$.
 
We now give the equations we are dealing with.
Set 
$s=n/2-1$. 
Consider the Cauchy problem on the scalar valued functions
$u$ and $V$ defined on $\mathbb{R}_+\times \mathbb{R}_x^n$
  \begin{align}
&
\partial_t u -\Delta u = div(u\nabla V)\label{eq1}\\
&\partial_{tt} V -\Delta V = u\label{eq2}\\
&u(0)= u_0\label{eq3}\\
&V(0)= V_0, V_t(0)= V_1\label{eq4}
\end{align}
 For $u_0\in {\dot{H}}^s$,
 $(\nabla V_0, V_1)\in {\dot{H}}^s
\times {\dot{H}}^s$ and 
 $u\in \tilde{L}^1_T
({\dot{H}}^s)$
 given, we denote by 
 $S(u, V_0, V_1)\in C^0\big(0, T, \mathscr{S}'(\mathbb{R}^n)\big)$
 the unique solution of the wave equation $\ref{eq2}$, $\ref{eq4}$. 
 With these 
 notations, the system 
 $\ref{eq1}-\ref{eq4}$ is interpreted
 as the following problem (P):
 
find 
$u\in \tilde{L}^1_T({\dot{H}}^s)$
 such that
  \begin{align}
&
\partial_t u -\Delta u = div(u\nabla S(u, V_0, V_1))\label{eq5}\\
&u(0)= u_0\label{eq6}
\end{align} 

For future reference, we recall a
standard result on the heat equation
(see \cite{bcd} p.157)
 \begin{proposition}\label{heat}
 Let $T> 0$, $\sigma \in \mathbb{R}^n$
 and $1\leq p\leq \infty$.
 Assume that $u_0\in 
{\dot{H}}^{\sigma}$ and
 $f\in\tilde{L}^p_T
 ({\dot{H}}^{\sigma-2+\frac{2}{p}})$. Then the problem
 \begin{align}
 &
\partial_t u -\Delta u = f\label{eq7}\\
&u(0)= u_0\label{eq8}
\end{align}
admits a unique solution 
$u\in
\tilde{L}^p_T
 ({\dot{H}}^{\sigma+\frac{2}{p}})\cap
 \tilde{L}^{\infty}_T
 ({\dot{H}}^{\sigma})$ and there
 exists $C> 0$ independent of $T$ such that,
 for any $q\in [p, \infty]$
 %for any $p \leq q\leq \infty$
 \begin{align}
 \Vert u
 \Vert_{\tilde{L}^q_T
 ({\dot{H}}^{\sigma+\frac{2}{q}})}
 \leq
 C\big(
 \Vert f
 \Vert_{\tilde{L}^p_T
 ({\dot{H}}^{\sigma-2+\frac{2}{p}})}
 +
  \Vert u_0
 \Vert_{
 {\dot{H}}^{\sigma}}
\big)
  \end{align}
Moreover, for $f=0$, we have $u\in C^0([0, T], {\dot{H}}^{\sigma})
\inj L^1([0, T], {\dot{H}}^{\sigma})$.

 The same statements hold true in nonhomogeneous Sobolev spaces
 with a constant $C = C_T$ depending
 on $T$. 
  \end{proposition}
In the sequel, we denote the solution
$u$ of proposition $\ref{heat}$ by
$$
u(t) = e^{t\Delta}u_0
 +
 \int_0^t e^{(t-\tau)\Delta}
 f(\tau)d\tau
$$
We will prove an existence result
for problem (P) by combining proposition
$\ref{heat}$ with the following ${\dot{H}}^s$
estimate for the solution 
$S(u, V_0, V_1)$ of the wave equation 
(see \cite{bcd} pp. 360-361)
 \begin{align}
  \Vert 
  \nabla S(u, V_0, V_1)
  \Vert_{\tilde{L}_T^{\infty}({\dot{H}}^s)}
  \leq
  C(
   \Vert 
  \nabla V_0
  \Vert_{{\dot{H}}^s}
  +
   \Vert  
   V_1
  \Vert_{{\dot{H}}^s}
  + \Vert
  u
  \Vert_{\tilde{L}^1_{T}({\dot{H}}^s)}
  )\label{stri}
 \end{align}
%%%%%%%%%%%%%%%%%%%%%%%%%%%%%%%%%%%%%%%%%%%%%%%%%%%%%%%%%%%%%%%
%%%%%%%%%%%%%%%%%%%%%%%%%%%%%%
\section{Existence and uniqueness in the 
case $n \geq 2$.}\label{exit}
%%%%%%%%%%%%%%%%%%%%%%%%%%%%%%%%%%%%%%%%%%%%%%%%%%%%%%%%%%%%%%%
%%%%%%%%%%%%%%%%%%%%%%%%%%%%%%%%%%%%%%%%%%%%%%%%%%%%%%%%%%%%%% 
 This part is devoted to the proof of existence of
 	a mild solution to problem (P).
 \begin{theorem}\label{pri}
  Let $n\geq 2$ and $s=n/2-1$. There exists
  $\eta>0$ such that, for any $T>0$, 
  $u_0\in {\dot{H}}^{s-2}$,
   $(\nabla V_0, V_1)\in {\dot{H}}^s
 \times {\dot{H}}^s$ with
 \begin{align}
 \Vert
  u_0
  \Vert_{{\dot{H}}^{s-2}}+
  \Vert 
  \nabla V_0
  \Vert_{{\dot{H}}^s}
  +
   \Vert  
   V_1
  \Vert_{{\dot{H}}^s} \leq 
  \eta\label{small}
 \end{align}
 there exists exactly one solution to the problem
 
 find $u\in \tilde{L}^1_T({\dot{H}}^s)$
 such that 
 $$
 u(t) = e^{t\Delta}u_0
 +
 \int_0^t e^{(t-\tau)\Delta}
 div\big(u\nabla S(u, V_0, V_1)\big)(\tau)d\tau
 $$
 When 
 %condition $\ref{small}$ is replaced 
 the above assumptions are replaced by
 $u_0\in {\dot{H}}^{s}$ and 
 $\Vert
 \nabla V_0
  \Vert_{{\dot{H}}^s}
  +
   \Vert  
   V_1
  \Vert_{{\dot{H}}^s} \leq 
  \eta$ small enough, we get local in time 
  existence and
 uniqueness.
 \end{theorem}

We will use the following classical lemma (see for instance \cite{bcd}
p.357). In this lemma, $\bar{B}(0, r)\subset E$ denotes the closed ball
of center $0$ and radius $r>0$.

\begin{lemma}\label{bil}
Let $E$ be a Banach space. Let
$\mathscr{B}: E\times E \rightarrow E$ be a continuous 
bilinear map and $\mathscr{L}: E \rightarrow E$
be a linear continuous map with $\Vert\mathscr{L}\Vert
<1$. Let $0<\alpha< (1-\Vert \mathscr{L}\Vert)^2/(4\Vert\mathscr{B}\Vert)$.
Then, for any $\gamma\in \bar{B}(0, \alpha)$, there exists 
exactly one $x\in\bar{B}(0, 2\alpha)$ such that 
$x = \gamma + \mathscr{L}(x)+\mathscr{B}(x,x)$.
\end{lemma}
In order to use lemma $\ref{bil}$, for $T>0$ and 
$(\nabla V_0, V_1)\in {\dot{H}}^s
 \times {\dot{H}}^s$  given, set 
 $E_T = \tilde{L}^1_T({\dot{H}}^s)$
and define $\mathscr{B}_T: E_T \times E_T \rightarrow E_T$
by 
\begin{align}
\mathscr{B}_T(u, w)= 
 \int_0^t e^{(t-\tau)\Delta}
 div\big(u\nabla S(w, 0, 0)\big)(\tau)d\tau\label{K1}
 \end{align}
 We also define $\mathscr{L}_T: E_T \rightarrow E_T$ by 
\begin{align}
\mathscr{L}_T(u)=
 \int_0^t e^{(t-\tau)\Delta}
 div\big(u\nabla S(0, V_0, V_1)\big)(\tau)d\tau\label{K2}
 \end{align}
Theorem
$\ref{pri}$ is an immediate consequence of lemma 
$\ref{bil}$ and the following $T$-independent estimates. 

\begin{lemma}\label{est}
   Let $n\geq 2$ and $s=n/2-1$, 
  $u_0\in {\dot{H}}^{s-2}$,
   $(\nabla V_0, V_1)\in {\dot{H}}^s
 \times {\dot{H}}^s$. Then, there exists
 $C_i>0$ $(0\leq i\leq 2)$ such that, for any $T>0$ and any 
 $u\in E_T$, $w\in E_T$, we have
 \begin{align}
  &\Vert \mathscr{B}_T(u, w)\Vert_{E_T}
  \leq 
  C_0 \Vert u \Vert_{E_T}
  \Vert w\Vert_{E_T}\label{beq1}\\
   &\Vert \mathscr{L}_T(u)\Vert_{E_T}
  \leq 
  C_1 
   \big( 
  \Vert 
  \nabla V_0
  \Vert_{{\dot{H}}^s}
  +
   \Vert  
   V_1
  \Vert_{{\dot{H}}^s}\big)
  \Vert u \Vert_{E_T}
 \label{beq2}\\
  &\Vert
  e^{t\Delta}u_0
  \Vert_{E_T}
  \leq C_2
  \Vert
  u_0
  \Vert_{{\dot{H}}^{s-2}}\label{beq3}
 \end{align}
\end{lemma}
\begin{proof}
 Inequality $\ref{beq3}$ follows from proposition
 $\ref{heat}$.
 Inequalities $\ref{beq1}$ and $\ref{beq2}$, amounts to
 $$\Vert
 \mathscr{B}_T(u, w)+
 \mathscr{L}_T(u)
 \Vert_{\tilde{L}_T^1({\dot{H}}^s)}
 \leq C
  \Vert u
 \Vert_{\tilde{L}_T^1({\dot{H}}^s)}
   \big(  \Vert w
 \Vert_{L_T^1({\dot{H}}^s)}+
  \Vert 
  \nabla V_0
  \Vert_{{\dot{H}}^s}
  +
   \Vert  
   V_1
  \Vert_{{\dot{H}}^s}\big)
 $$
 Set $$z =  \mathscr{B}_T(u, w)+
 \mathscr{L}_T(u)=  \int_0^t e^{(t-\tau)\Delta}
 div\big(u\nabla S(w, V_0, V_1)\big)(\tau)d\tau
 $$
Proposition $\ref{heat}$ provides
 $$\Vert z
  \Vert_{\tilde{L}_T^1({\dot{H}}^s)}
  \leq C
  \Vert div(u\nabla S(w, V_0, V_1))
  \Vert_{\tilde{L}_T^1({\dot{H}}^{s-2})}
 $$
 hence 
 \begin{align}\label{pok}
  \Vert z
  \Vert_{\tilde{L}_T^1({\dot{H}}^s)}
  \leq C
  \Vert u\nabla S(w, V_0, V_1)
  \Vert_{\tilde{L}_T^1({\dot{H}}^{s-1})}
 \end{align}
Since $-n/2<s<n/2$, the product is continuous from 
$\tilde{L}_T^1({\dot{H}}^s)
\times \tilde{L}_T^{\infty}({\dot{H}}^s)$ to 
$\tilde{L}_T^1({\dot{H}}^{2s-n/2}) = 
\tilde{L}_T^1({\dot{H}}^{s-1})$.
See for instance \cite{bcd} pages 90 and 98 or use Bony's decomposition.
With $\ref{pok}$, this implies that
$$
 \Vert z
  \Vert_{\tilde{L}_T^1({\dot{H}}^s)}
  \leq C
  \Vert u\Vert _{\tilde{L}_T^1({\dot{H}}^s)}
  \Vert \nabla S(w, V_0, V_1)
  \Vert_{\tilde{L}_T^{\infty}({\dot{H}}^s)}
$$
and with $\ref{stri}$
 $$\Vert
 z
 \Vert_{\tilde{L}_T^1({\dot{H}}^s)}
 \leq C
  \Vert u
 \Vert_{\tilde{L}_T^1({\dot{H}}^s)}
   \big(  \Vert w
 \Vert_{\tilde{L}_T^1({\dot{H}}^s)}+
  \Vert 
  \nabla V_0
  \Vert_{{\dot{H}}^s}
  +
   \Vert  
   V_1
  \Vert_{{\dot{H}}^s}\big)
 $$
\end{proof}
 
Global existence and uniqueness
in theorem $\ref{est}$ is a consequence
of lemmas $\ref{bil}$ and $\ref{est}$ by 
restricting to small data, i.e
$\Vert 
  \nabla V_0
  \Vert_{{\dot{H}}^s}
  +
   \Vert  
   V_1
  \Vert_{{\dot{H}}^s}<1/C_1$ and 
  $\Vert
  u_0
  \Vert_{{\dot{H}}^{s-2}}<
  [1-C_1
  (\Vert 
  \nabla V_0
  \Vert_{{\dot{H}}^s}
  +
   \Vert  
   V_1
  \Vert_{{\dot{H}}^s})]^2/(4C_0C_2)
  $.
The local existence and uniqueness part
is a consequence of the same lemmas once 
$\lim_{t\rightarrow 0}\Vert e^{\tau\Delta}u_0\Vert_{E_t}
=0$ is proved. Recalling the assumption 
$u_0\in {\dot{H}}^{s}$, this follows from the fact that
$e^{\tau\Delta}u_0\in L^1([0, t], {\dot{H}}^s
(\mathbb{R}))$ (see proposition $\ref{heat}$) and the inequality
(see \cite{bcd} p.98)
$\Vert e^{\tau\Delta}u_0\Vert_{E_t}\leq
\Vert e^{\tau\Delta}u_0\Vert_{L^1([0, t], {\dot{H}}^s
(\mathbb{R}))}
\rightarrow 0$ when $t\rightarrow 0
$.

 \begin{remark}
 The proof of theorem
 $\ref{pri}$ extends
 to the Debye type system (see \cite{bwn},
  \cite{kn}): $\partial_t u_j -\Delta u_j = div(\beta_j u_j\nabla V)$,
 $\partial_{tt} V -\Delta V = \sum_{k}\alpha_ku_k$, 
  $u_j(0)= u_{j,0}$, $V(0)= V_0, V_t(0)= V_1$ with 
  $(\alpha_j, \beta_j)
 \in\mathbb{R}^2$ $(1\leq j\leq m)$ given.
  \end{remark}
%%%%%%%%%%%%%%%%%%%%%%%%%%%%%%%%%%%%%%%%%%%%%%%%%%%%%%%%%%%%%%%
%%%%%%%%%%%%%%%%%%%%%%%%%%%%%%
\section{Existence and uniqueness in the case $n = 1$.}\label{prelim}
%%%%%%%%%%%%%%%%%%%%%%%%%%%%%%%%%%%%%%%%%%%%%%%%%%%%%%%%%%%%%%%
%%%%%%%%%%%%%%%%%%%%%%%%%%%%%%%%%%%%%%%%%%%%%%%%%%%%%%%%%%%%%%
Until the end of the paper, $n=1$.  
We still denote by 
 $S(u, V_0, V_1)\in C^0\big(0, T, \mathscr{S}'(\mathbb{R}^n)\big)$
 the unique solution of the wave equation $\ref{eq2}$, $\ref{eq4}$, and
$\mathscr{B}_T(u, w)$ and $\mathscr{L}_T(u)$ are still
formally defined by formulas 
$\ref{K1}$ and $\ref{K2}$.
The notation $L^p_x$ stands for
$L^p(\mathbb{R}_x)$. Last, $\nabla = div =\partial_x$.

As a building block in the proof 
of the existence theorem $\ref{n1}$
below, we first establish a $L^1$
estimate for solutions of equations
$\ref{eq5}$, $\ref{eq6}$ (lemma $\ref{L1}$). 
We begin with two simple trace-lemmas. For 
$y\in \mathbb{R}$, set $D_T(y) = ]0, T[
\times ]y, y+1[$ and $\bar{D}_T(y) = [0, T]
\times [y, y+1]$.
\begin{lemma}\label{TSSA}
 Let $T>0$, $y\in\mathbb{R}$. There exists
 $C > 0$ such that, for any $y\in \mathbb{R}$
 and $f\in 
  C^1(\bar{D}_T(y))$, we have 
 \begin{align}\label{VGE}
 \Vert f(., y)\Vert_{L^2(0, T)}
 +
  \Vert f(., y+1)\Vert_{L^2(0, T)}
  \leq
  C \Vert f \Vert_{L^2(0, T, H^1(]y, y+1[))}
 \end{align}
 \end{lemma}
 \begin{proof}
  We only prove inequality
   \begin{align}\label{KER}
    \int_0^T\vert f(\tau, y)\vert^2d\tau
  \leq C \Vert f \Vert_{L^2(0, T, H^1(]y, y+1[))}^2
   \end{align}
   Let $\phi \in C^1(\bar{D}_T(0))$ with
   $\phi(t, x) = 1$ for $(t, x)\in [0, T]\times[0, 1/4]$
   and 
    $\phi(t, x) = 0$ for $(t, x)\in [0, T]\times[3/4, 1]$.
   For $y\in\mathbb{R}$ fixed, 
    define $\phi^y\in C^1(\bar{D}_T(y))$
    by $\phi^y(t, x)=\phi(t, x-y)$.
    Let $f\in C^1(\bar{D}_T(y))$. We have
 \begin{align}
  \int_0^T\vert f\vert^2(\tau, y) d\tau
  =\int_0^T\vert \phi^yf\vert^2(\tau, y) d\tau
  &\leq \int_0^T
  \int_y^{y+1}2
  \vert \phi^yf(\phi^yf)_x\vert(\tau, s)ds
  d\tau\nonumber\\
  &\leq 2
  \Vert \phi^y \Vert_{W^{1, \infty}(D(y))}^2
  \Vert
  f
  \Vert^2_{L^2(0, T, H^1(]y, y+1[))}
  \label{PL1}
 \end{align}
 Since $\Vert \phi^y \Vert_{W^{1, \infty}(D(y))}
 =\Vert \phi \Vert_{W^{1, \infty}(D(0))}$, we get inequality
 $\ref{KER}$.
 \end{proof}
 Let $y\in\mathbb{R}$. By lemma 
 $\ref{TSSA}$, we can define
two continuous trace-operators
 $\gamma_{y+1}^{-}$ and  
 $\gamma_{y}^{+}:L^2(0, T, H^1(\mathbb{R}))\rightarrow
 L^2(0, T)$ by 
 $
\gamma_y^+(f)(\tau) = f(\tau, y) 
\textrm{ and } \gamma_{y+1}^{-}(f)(\tau) = f(\tau, y+1)$
for any 
$f\in C^1(\bar{D}_T(y))$. For future reference,
notice that,  
 for any
$f\in L^2(]0, T[, H^1(]y, y+1[))$, we have 
 \begin{align}\label{VGF} 
 &\Vert \gamma_{y+1}^{-}(f)\Vert_{L^1(0, T)}
  +
  \Vert\gamma_y^+(f)\Vert_{L^1(0, T)}\nonumber\\
  &\leq \sqrt{T}(\Vert \gamma_{y+1}^{-}(f)\Vert_{L^2(0, T)}
  +
  \Vert\gamma_y^+(f)\Vert_{L^2(0, T)})
  \leq C \sqrt{T}\Vert f \Vert_{L^2(0, T, H^1(]y, y+1[))}
 \end{align}
 with a constant $C>0$ independant of $y\in \mathbb{R}$.
  
 The second lemma is a consequence
 of the continuity of the trace functions
 $\gamma_y^{\pm}$ and density arguments.
 We only prove inequality $\ref{GREE2}$.
 In the sequel $sign$ denotes the sign function.
%  We only prove $\ref{GREE2}$.
 \begin{lemma}\label{CSA}
  Let $T>0$ and $y\in\mathbb{R}$. Let also
  $f\in L^2(0, T, H^2(\mathbb{R}))$,
  $g\in L^{\infty}(0, T, H^2(\mathbb{R}))$
  and $\phi \in L^{2}(0, T, H^1(\mathbb{R}))$.
  Then, for any $(y, z) \in \mathbb{R}^2$,           
  $y<z$, and any $t\in[0, T]$ 
  we have
  
  a) \begin{align}
  \int_0^t\int_y^z  \phi \partial_{xx}  fdxd\tau
  = -
  \int_0^t\int_y^z \partial_x\phi\partial_x f
  dxd\tau
  +
  \int_0^t[\gamma_z^-( \phi\partial_x f)
  -\gamma_y^+(\phi\partial_x f)]d\tau\label{GREE1} 
   \end{align}
   \begin{align}
  \int_0^t\int_y^z sign(f)\partial_x(f\partial_x g)
  dxd\tau
  =\int_0^t [\gamma_z^-(\vert f
  \vert \partial_xg) - \gamma_y^+(\vert f
  \vert \partial_x g)]d\tau\label{GREE2}
   \end{align}
  
  b) Let 
  $A\in L^{\infty}(0, T, H^1(\mathbb{R}))$.
  Then $\vert \gamma^{\pm}_y(Af)\vert (\tau)
  \leq 
  \Vert A\Vert_{L^{\infty}(0, T, H^1(]y, y+1[))}
  \vert\gamma^{\pm}_y(f)\vert(\tau)$ for 
  almost every $\tau\in [0, T]$.
 \end{lemma}
 \begin{proof}(of $\ref{GREE2}$). We first prove that, for any 
$(F, G)\in L_T^2(H^1)\times 
L_T^{\infty}(H^2)$, we have 
\begin{align}
\int_0^t\int_y^z \partial_x(F\partial_x G) dxd\tau
  =\int_0^t [\gamma_z^-(F
  \partial_xG) - \gamma_y^+(F
  \partial_xG)]d\tau\label{OUM}
  \end{align}
  The case 
  $(F, G)\in C^{\infty}([0, T]\times
  \mathbb{R})^2$ is obvious. Since both 
  sides of $\ref{OUM}$ are 
  continuous with respect to 
  $(F, G)\in L_T^2(H^1)\times 
L_T^{\infty}(H^2)$, owing to  the fact that 
$H^1(\mathbb{R})$ is a Banach algebra and the continuity $\ref{VGF}$
of the trace functions,
we get equality
$\ref{OUM}$
by density. 
Now, for 
 $f\in L^2(0, T, H^1(\mathbb{R}))$,
  $g\in L^{\infty}(0, T, H^2(\mathbb{R}))$,
  we have $\vert f \vert\in L^2(0, T, H^1(\mathbb{R}))$ and 
  $\partial_x(f\partial_x g) sg(f) =
  \partial_x(\vert f\vert\partial_x g)$. Therefore, equality $\ref{GREE2}$ follows from
  equality
$\ref{OUM}$ with $F = \vert f \vert$
and $G = g$.
\end{proof}
We are ready to prove our main $L^1$ lemma.
\begin{lemma}\label{L1}
 Let $T>0$, $V_0\in H^2(\mathbb{R})$, $V_1\in H^1(\mathbb{R})$,
$u_0\in H^1(\mathbb{R})
 \cap L^1(\mathbb{R})$ and $u\in 
L_T^2(H^2) \cap H_T^1(L^2) \cap C^0([0, T], 
 H^1)$. Assume that
  $S(u, V_0, V_1)\in C^0([0,
T], H^2)\cap C^1([0,
T], H^1)$ and assume that 
function $(u,S(u, V_0, V_1))$ is a solution
of equations $\ref{eq5}$, $\ref{eq6}$, 
i.e 
$$u(t) = e^{t\Delta}u_0
 +
 \int_0^t e^{(t-\tau)\Delta}
 div\big(u\nabla S(u, V_0, V_1)\big)(\tau)d\tau
 $$
Then
$u\in C^0([0, T], 
 L^1(\mathbb{R}))$ and $\Vert u(t) \Vert_{L^1} \leq
  \Vert u_0 \Vert_{L^1}$, for any 
 $t\in [0, T]$.
  Moreover, when $\pm u_0 \geq 0$, we
   have $\pm u\geq 0$ and $\Vert u(t) \Vert_{L^1}
   = \Vert u_0\Vert_{L^1}$.
\end{lemma}
\begin{proof} Let $(y, z)\in \mathbb{R}^2$
with $y<z$, and $t\in[0, T]$.  Since $u\in 
L_T^2(H^2)$ and 
$S(u, V_0, V_1)\in C^0([0,
T], H^2)$, we can apply formula
$\ref{GREE2}$ with $f = u$,
$g = S(u, V_0, V_1)$.
Hence, multiplying $\ref{eq5}$ by $sign(u)$ and
integrating on $[0, T]\times [y, z]$, we obtain
\begin{align}
   &\int_y^z \vert u \vert(t, x)dx
   = 
   \Vert u_0 \Vert_{L^1([y, z])}
   +
   \int_0^t [\gamma_z^-(\vert u \vert\partial_xS) - 
   \gamma_y^+(\vert u \vert\partial_xS)]
   d\tau
     +
   \int_0^t\int_y^z sign(u)\Delta udxd\tau
   \nonumber\\
   &\leq 
   \Vert u_0 \Vert_{L^1([y, z])}
   +
   \Vert\gamma_z^-(\vert u \vert\partial_xS) \Vert_{L^1([0, T])}
   +
   \Vert\gamma_y^+(\vert u \vert\partial_xS) \Vert_{L^1([0, T])}
   %\nonumber\\
   +
   \sup_{t\in[0, T]}
   \Big(
   \int_0^t\int_y^z sign(u)\Delta udxd\tau
   \Big)\label{MIK}
    \end{align}
    We majorize the last three terms
    in inequality $\ref{MIK}$. Using the 
    inequality $\ref{VGF}$, we get
    \begin{align}
    \Vert\gamma_z^-(\vert u \vert\partial_xS) \Vert_{L^1([0, T])}+
    \Vert\gamma_y^+(\vert u \vert\partial_xS) \Vert_{L^1([0, T])}\leq 
     C \sqrt{T}
    \Vert \vert u \vert\partial_xS
    \Vert_{L^2(0, T, H^1(]y, y+1[\cup]z-1, z[))}
    \label{VAPS}
    \end{align}
     Hence,  $\ref{VAPS}$ and
     $\vert u \vert\partial_xS
\in L^2(0, T, H^1(\mathbb{R}))$
     implies that
     \begin{align}
     \lim_{inf(\vert y \vert, 
 \vert z\vert) \rightarrow \infty}
    \big( \Vert\gamma_z^-(\vert u \vert\partial_xS) \Vert_{L^1([0, T])}
     +
     \Vert\gamma_y^+(\vert u \vert\partial_xS) \Vert_{L^1([0, T])} \big)= 0\label{VAB}
      \end{align}
  We next prove that 
 \begin{align}
\limsup_{inf(\vert y \vert, 
 \vert z\vert) \rightarrow \infty}
 \Big(\sup_{t\in [0, T]}
 \int_0^t\int_y^z sign(u)\Delta udxd\tau\Big)\leq0 \label{VAD}
  \end{align}
Set $h_{\epsilon}(x) = x/\sqrt{x^2+\epsilon}$ ($x\in \mathbb{R}$,
$\epsilon >0$). Due
to $\Delta u \in L^2([0, T]\times\mathbb{R})
\inj
L^1_{loc}([0, T]\times\mathbb{R})$, 
$\Vert h_{\epsilon}\Vert_{L^{\infty}}\leq 1$
and Lebesgue theorem, we have
\begin{align}
\int_0^t\int_y^z sign(u)\Delta udxd\tau
= 
\lim_{\epsilon\rightarrow 0}
\int_0^t\int_y^z h_{\epsilon}(u)
\Delta udxd\tau\label{NAZE}
\end{align}
Using 
$\ref{GREE1}$ with $f=u\in
L_T^2(H^2)$, $\phi = 
h_{\epsilon}(u)\in
L_T^2(H^1)$, majorizing, and appealing to
lemma $\ref{CSA}$ b), we obtain an
$(\epsilon, t)$-independent estimate 
   \begin{align}
   \int_0^t\int_y^z \Delta u h_{\epsilon}(u)dxd\tau
    &= -\int_0^t\int_y^z \vert \nabla u\vert^2
   h_{\epsilon}'(u)dxd\tau
   +\int_0^t[\gamma_z^{-}
   (h_{\epsilon}(u)\nabla u)-
     \gamma_y^{+}(
     h_{\epsilon}(u)\nabla u)]
     d\tau\nonumber\\
   &\leq
   \int_0^T[\vert \gamma_z^{-}
   (\nabla u) \vert+
   \vert \gamma_y^{+}
   (\nabla u) \vert]d\tau\label{mich}
      \end{align}
      Appealing to $\ref{VGF}$,
      we deduce from $\ref{mich}$ and $\ref{NAZE}$ that 
       \begin{align}
        \sup_{t\in [0, T]}\Big(
 \int_0^t\int_y^z sign(u)\Delta u (\tau, x)dxd\tau\Big)
 \leq C \sqrt{T}
  \Vert \nabla u 
   \Vert_{L^2(0, T, H^1(]y, y+1[\cup]z-1,z[)}\label{BJOUR}
       \end{align}   
    Hence, inequality $\ref{VAD}$ follows
    from $\ref{BJOUR}$ and
    $u \in 
   L^2(0, T, H^2(\mathbb{R}))$.  
   Set $y = -z$ 
   and let $z\rightarrow +\infty$
   in $\ref{MIK}$. Using  $\ref{VAB}$, 
   $\ref{VAD}$ and 
   the monotone convergence theorem, we get
   $\Vert u(t) \Vert_{L^1}
   \leq \Vert u_0\Vert_{L^1}$ for $t\in 
   [0, T]$. Replacing the sign function by
   the negative or positive part functions 
   $(.)^{\pm}$,
   or the
   constant function $1$, the same argument
   provides $\Vert (u)^{\pm}(t) \Vert_{L^1}
   \leq \Vert (u_0)^{\pm}\Vert_{L^1}$
   and $\int_{\mathbb{R}}u(t, x)dx = 
   \int_{\mathbb{R}}u_0(x)dx$. In the 
   particular case $\pm u_0 \geq 0$, we
   recover $\pm u(t)\geq 0$ and $\Vert u(t) 
   \Vert_{L^1}
   = \Vert u_0\Vert_{L^1}$ for any $t\in 
   [0, T]$.
   
  Finally, appealing to 
   $\ref{MIK}$, $\ref{VAB}$,
   $\ref{VAD}$ and $u_0\in L^1$, we have
    $sup_{\tau\in [0, T]}\Vert 
   u(\tau)\Vert_{L^1(]-\infty, -x[\cup
   ]x, +\infty[)}\rightarrow 0$ when
   $x\rightarrow +\infty$. Hence, noticing that
   $u\in 
   C^0(0, T, L^2(\mathbb{R}))
   \inj
   C^0(0, T, L^1_{loc}(\mathbb{R}))$, we easily
   obtain $u\in C^0(0, T, L^1(\mathbb{R}))$
\end{proof}
It follows that
\begin{theorem}\label{n1}
 Let $T>0$, $V_0\in H^2(\mathbb{R})$, $V_1\in H^1(\mathbb{R})$
 and $u_0\in H^1(\mathbb{R})
 \cap L^1(\mathbb{R})$. The problem
 
find $u\in L_T^2(H^1)$ such that
 $$
 u(t) = e^{t\Delta}u_0
 +
 \int_0^t e^{(t-\tau)\Delta}
 div\big(u\nabla S(u, V_0, V_1)\big)(\tau)d\tau
 $$
 admits exactly one solution.
 Moreover, $u \in 
L_T^2(H^2) \cap H_T^1(L^2) \cap C^0([0, T], 
 H^1 \cap L^1)$ and $\Vert u(t) \Vert_{L^1_x} \leq
  \Vert u_0 \Vert_{L^1}$.
  Last, when $\pm u_0 \geq 0$, we
   have $\pm u\geq 0$ and $\Vert u(t) \Vert_{L^1}
   = \Vert u_0\Vert_{L^1}$, $0\leq t\leq T$.
\end{theorem}
\begin{proof}
Step 1 (local  existence). For any $0<T<1$, 
$u\in L_T^2(H^1)$ and 
$w\in L_T^2(H^1)$, 
proposition $\ref{heat}$ provides
\begin{align}
\Vert
\mathscr{B}_T(u, w)+\mathscr{L}_T(u)
\Vert_{L_T^2(H^1)}
&\leq
C 
\Vert
 div\big(u\nabla S(w, V_0, V_1)\big)
\Vert_{L_T^1(L^2)}\nonumber\\
&\leq C
\Vert
u\nabla S(w, V_0, V_1)
\Vert_{L_T^1(H^1)}\nonumber\\
&\leq C
\Vert
u
\Vert_{L_T^2(H^1)}
\Vert
\nabla S(w, V_0, V_1)
\Vert_{L_T^2(H^1)}\label{ener}
\end{align}
 Note that
 \begin{align}
 S(w, V_0, V_1)(t, x) = 
 \int_0^t\int_{x-(t-\tau)}^{x+(t-\tau)}
 w(\tau, \xi)d\xi d\tau
 +\frac{1}{2} \big(V_0(x+t)+V_0(x-t)
 +\int_{x-t}^{x+t}V_1(s)ds\big)\label{hip}
 \end{align}
 Assume first that $V_0$, $V_1$ and 
 $u$ are infinitely differentiable. We have
 \begin{align}
 \nabla S(w, V_0, V_1)(t, x)
 =
 &\int_0^t
 \big(w(\tau, x+(t-\tau))
 -
  w(\tau, x-(t-\tau)\big)
 d\tau\nonumber\\
 &+\frac{1}{2} \big(V_0^{'}(x+t)+V_0^{'}(x-t)
 +V_1(x+t)
 -V_1(x-t)
 \big)\label{hap}
 \end{align}
 Arguing by density, the formula 
 $\ref{hap}$ holds 
for the distributional derivative
$\nabla S$  under the assumptions of theorem  $\ref{n1}$, and we have
 \begin{align}
 \Vert
\nabla S(w, V_0, V_1)
\Vert_{L_T^2(H^1)}
\leq 
C
 \sqrt{T}
 \big(
 \Vert
 w
 \Vert_{L_T^2(H^1)}
 +
 \Vert
 V_0^{'}
 \Vert_{H^1}
 + \Vert
 V_1
 \Vert_{H^1}
 \big)\label{hup}
 \end{align}
 From $\ref{ener}$ and $\ref{hup}$
 we deduce that
  \begin{align}
  \Vert
\mathscr{B}_T(u, w)+\mathscr{L}_T(u)
\Vert_{L_T^2(H^1)}
&\leq
C \sqrt{T}
\Vert
u
\Vert_{L_T^2(H^1)}
\big(
\Vert
 w
 \Vert_{L_T^2(H^1)}
 +
 \Vert
 V_0^{'}
 \Vert_{H^1}
 + \Vert
 V_1
 \Vert_{H^1}
\big)\label{bip}
   \end{align}
   Finally, notice that
    \begin{align}
   \Vert
e^{t\Delta}u_0
\Vert_{L_T^2(H^1)}\leq
\Vert
u_0
\Vert_{L^2}\label{bup}
 \end{align}
%\end{proof}
Hence, for $T> 0$ small enough,
the existence and uniqueness of a solution
$u\in L_T^2(H^1)$ follows from lemma
$\ref{bil}$, inequalities 
$\ref{bip}$ and $\ref{bup}$.

Next, 
since $u \in L_T^2(H^1)$,
$V_0 \in H^2$ and $V_1 \in H^1$
we get 
\begin{align}
S(u, V_0, V_1)\in C^0([0,
T], H^2)\cap C^1([0,
T], H^1)\label{reg1}
\end{align}
 Therefore
$div(u\nabla S(u, V_0, V_1))\in L_T^2(L^2)$.
Due to $u_0 \in H^1$, equation
$\partial_t u - \Delta u 
= div(u\nabla S(u, V_0, V_1))$, 
proposition $\ref{heat}$ (and by interpolation),
we thus obtain 
\begin{align}
u \in 
L_T^2(H^2) \cap H_T^1(L^2) \cap C^0([0, T], 
 H^1)\label{reg2}
\end{align}
%Step 2 (the $L^1$ estimate). 
%Let $T > 0$ as defined in step 1. The $L^1$
Step 2 (global existence). Let $T^* >0$,  the maximal time of existence 
 of a mild solution 
  endowed with the properties $\ref{reg1}$, 
 $\ref{reg2}$. In particular,
 $u\in C^0([0, T], L^1(\mathbb{R}))$ 
 and $\Vert u(t) \Vert_{L^1_x} \leq
 \Vert u_0 \Vert_{L^1}$ for $0\leq t\leq T
 <T^*$ (see lemma $\ref{L1}$).
 In order to prove that $T^* =\infty$,
 and due  to $\ref{reg1}$, $\ref{reg2}$ 
 and lemma $\ref{L1}$, we essentially 
 have to find an a priori estimate on
 $\Vert
 u
 \Vert_{L^2_{T}(H^1)}$
 for $0< T< T^*$. 
 We multiply equation 
 $\ref{eq5}$ by $u$ and integrate we respect to
 $x$. Appealing to $\ref{reg1}$ and $\ref{reg2}$,
 we get
 \begin{align}
 \frac{1}{2}\frac{d}{dt}\Vert u
 \Vert_{L_x^2}^2
 +\int_{\mathbb{R}}\vert 
 \partial_xu\vert^2dx
 &\leq
 \vert
 \int_{\mathbb{R}}
 u \partial_xu \partial_xSdx
 \vert\nonumber\\
  &\leq
  \Vert u
 \Vert_{L_x^4}
 \Vert \partial_xu
 \Vert_{L_x^2}
 \Vert \partial_xS
 \Vert_{L_x^4}\label{flip}
  \end{align}
We now estimate 
$ \Vert u
 \Vert_{L_x^4}$ and $\Vert \partial_xS
 \Vert_{L_x^4}$. By Gagliardo-Nirenberg
 inequalities and the $\textrm{L}^1$ 
 properties of $u$,
 we have
 \begin{align}
 \Vert u
 \Vert_{L_x^4}^4
 \leq C
 \Vert u
 \Vert_{L_x^1}^2
 \Vert \partial_xu
 \Vert_{L_x^2}^2
 \leq 
 C
 \Vert u_0
 \Vert_{L^1}^2
 \Vert \partial_xu
 \Vert_{L_x^2}^2\label{GN}
 \end{align}
 Using equation $\ref{hap}$ with
 $w=u$, we obtain for any $0< t< T$
  \begin{align}
  \Vert \partial_xS(t)
 \Vert_{L^4}
 \leq
 C\big(
  \Vert u
 \Vert_{L^{1}(0, t, L^4)}
 + \Vert V_{0}^{'}
 \Vert_{L^4}
 +\Vert V_{1}
 \Vert_{L^4}
 \big)
 \label{sta}
   \end{align}
   Invoking inequality $\ref{GN}$, this implies that
 \begin{align} 
  \Vert \partial_xS(t)
 \Vert_{L^4}
 &\leq
 C\big(
 \Vert u_0
 \Vert_{L^1}^{1/2}
 \int_0^t\Vert \partial_xu(\tau)
 \Vert_{L^2}^{1/2}d\tau
 + \Vert V_{0}^{'}
 \Vert_{L^4}
 +\Vert V_{1}
 \Vert_{L^4}
 \big)\nonumber\\
 &\leq
 C\big(
 \Vert u_0
 \Vert_{L^1}^{1/2}
 \Vert \partial_xu
 \Vert_{L^2(]0, T[\times\mathbb{R})}^{1/2}
 t^{3/4}
 + \Vert V_{0}^{'}
 \Vert_{L^4}
 +\Vert V_{1}
 \Vert_{L^4}\big)
 \label{stuc}
   \end{align} 
   Therefore, $\ref{flip}$, 
   $\ref{GN}$, $\ref{stuc}$ and injection
   $H^{1}(\mathbb{R})\inj 
   L^4(\mathbb{R})$ provide
    \begin{align}
 \frac{1}{2}\frac{d}{dt}\Vert u
 \Vert_{L_x^2}^2
 +
  \Vert \partial_xu(t)
 \Vert_{L_x^2}^{2}
 &\leq C
  \Vert u_0
 \Vert_{L^1}^{1/2}
 \Vert \partial_xu(t)
 \Vert_{L^2}^{3/2}
 \big(
 \Vert u_0
 \Vert_{L^1}^{1/2}
 \Vert \partial_xu
 \Vert_{L^2(]0, T[\times\mathbb{R})}^{1/2}
 t^{3/4}
 + \Vert V_{0}^{'}
 \Vert_{L^4}
 +\Vert V_{1}
 \Vert_{L^4}\big)\nonumber\\
 &\leq
 \eta 
   \Vert u_0
 \Vert_{L^1}^{2/3}
 \Vert \partial_xu(t)
 \Vert_{L^2}^{2}\nonumber\\
 &\phantom{HHHHH}+
 C_{\eta} 
 \big(t^{3}\Vert u_0
 \Vert_{L^1}^{2}
 \Vert \partial_xu
 \Vert_{L^2(]0, T[\times\mathbb{R})}^{2}
 + \Vert V_{0}
 \Vert_{H^2}^4
 +\Vert V_{1}
 \Vert_{H^1}^4
 \big)
 \label{flac}
  \end{align}
  Take $\eta$ such that $0<\eta\Vert u_0\Vert_{L^1}^{2/3}\leq 1/2$.
  Hiding the term
  $\eta 
   \Vert u_0
 \Vert_{L^1}^{2/3}
 \Vert \partial_xu(t)
 \Vert_{L^2}^{2}$ in the left
 hand side of $\ref{flac}$, setting 
 $z(t) = \Vert u(t)
 \Vert_{L^2}^2
 +
 \int_{0}^{t}
  \Vert \partial_xu(\tau)
 \Vert_{L^2}^{2}d\tau$ and using the
 Gronwall inequality we obtain the required 
 a priori estimate on 
 $\Vert
 u
 \Vert_{L^2_{T}(H^1)}$. The end of the proof
 is standard and we omit further details. 
 \end{proof}
 %%%%%%%%%%%%%%%%%%%%%%%%%%%%%%%%%%%%%%
%\section*{References}

\end{document}